\documentclass{amsart}
\usepackage{amsmath}
\usepackage{amsthm}
\usepackage{amssymb}
\usepackage{mathrsfs}
\usepackage{graphicx}
\usepackage{caption}
\usepackage{subcaption}
\usepackage{xcolor}
\usepackage{xcomment}

\newtheorem{theorem}{Theorem}[section]
\newtheorem{lemma}[theorem]{Lemma}
\theoremstyle{definition}
\newtheorem{definition}[theorem]{Definition}

\newtheorem{corollary}[theorem]{Corollary}
\newtheorem{proposition}[theorem]{Proposition}
\newtheorem{remark}[theorem]{Remark}

\numberwithin{equation}{section}
\newcommand{\R}{\mathbb{R}}

\newcommand{\B}{\mathcal{B}}
\newcommand{\F}{\mathscr{F}}

\newcommand{\Z}{\mathbb{Z}}
\newcommand{\Hi}{\mathcal{H}}
\newcommand*{\bigchi}{\mbox{\Large$\chi$}}
\def\Msp{{\boldsymbol M}}

\begin{document}

\title[{Modulation spaces, multipliers associated with the SAFT}]{Modulation spaces, multipliers associated with the special affine Fourier transform}

\author{M. H. A. Biswas}
\address{Department of Mathematics, Indian Institute of Technology Madras, Chennai - 600036, India}
\email{mdhasanalibiswas4@gmail.com}

\author{H. G. Feichtinger}
\address{Faculty of Mathematics, University of Vienna, Oskar-Morgenstern-Platz 1, A-1090 Wien, Austria}
\email{hans.feichtinger@univie.ac.at}
\author{R. Ramakrishnan}
\address{Department of Mathematics, Indian Institute of Technology Madras, Chennai - 600036, India}
\email{radharam@iitm.ac.in}



\begin{abstract}
We study some fundamental properties of the special affine Fourier transform (SAFT) in connection with the Fourier analysis and time-frequency analysis. We introduce the modulation space $\Msp^{r,s}_A$ in connection with SAFT and prove that if a bounded linear operator between new modulation spaces commutes with $A$-translation, then it is a $A$-convolution operator. We also establish H\"{o}rmander multiplier theorem and Littlewood-Paley theorem associated with the SAFT.
\end{abstract}
\maketitle
\section{Introduction}

\par The special affine Fourier transform (SAFT) was first studied by Abe and Sheridan in \cite{Abe} in connection with optical wave functions. It is an integral transform acting on the optical wave function and is related to the special affine linear transform of the phase space
\begin{equation}\label{eq:Intro1}
\begin{pmatrix}
t^\prime\\\omega^\prime
\end{pmatrix}
=\begin{pmatrix}
a&b\\
c&d
\end{pmatrix}
\begin{pmatrix}
t\\ \omega
\end{pmatrix}
+
\begin{pmatrix}
p \\ q
\end{pmatrix},
\end{equation}
where all the parameters $a,b,c,d,p,q$ are real and $ad-bc=1$. The integral representation of the wave function transformation connected with (\ref{eq:Intro1}) defines the special affine Fourier transform given as follows. For $f\in L^1(\R)$,
\textcolor{black}{\begin{equation}
\F_Af(\omega)=\frac{1}{\sqrt{|b|}}\int_\R f(t)e^{\frac{\pi i}{b}(at^2+2pt-2\omega t+2(bq-dp)\omega+d\omega^2)}dt, ~\omega \in \R,
\end{equation}}
where $A$ stands for the set of parameters $\{a,b,c,d,p,q\}$, with $b\neq 0$. When $p=q=0,$ it is called the linear canonical transform. For $A=\{cos\theta, sin\theta, -sin\theta, cos\theta, 0,0\} $ it is called the fractional Fourier transform, and for $A=\{1,b,0,1,0,0\}$, it is known as the Fresnel transform.

\par In \cite{Bhandari}, Bhandari and Zayed studied shift invariant spaces associated with the SAFT using chirp modulation. Recently in \cite{H2}, Filbir et al defined a new translation operator $T_x^A$ and studied the corresponding shift invariant spaces in the context of the SAFT. This new translation operator leads to $A$-convolution and is used to generalize Wendel's theorem for multipliers associated with the SAFT. Recently in \cite{GrafakosAcha21}  Chen et al studied variants of the  H\"{o}rmander multiplier theorem and of the Littlewood-Paley theorem associated with fractional Fourier transforms.

\par Modulation spaces were introduced by Feichtinger in \cite{fe83-4} and published in  \cite{FeichtingerLCA}. Meanwhile they have become a universal
tool for various branches of analysis, including pseudo-differential
operators, harmonic analysis in general and specifically time-frequency
and Gabor analysis. 
 The well known spaces such as weighted $L^2$ spaces, Bessel potential spaces, Shubin-Sobolev spaces are important examples of modulation spaces. (See \cite{BoggiElenaGro04}, \cite{Luef11}). \textcolor{black}{There are certain interesting embedding results among modulation spaces. We refer to (\cite{embedding05}, \cite{embedding04}, \cite{embedding12}) in this connection. Further, using convolution, it has been shown in \cite{toft04} that $\Msp^{p_1,q_1}\star \Msp^{p_2,q_2}\subset \Msp^{p_0,q_0},$ where $1/p_1+1/p_2=1+1/p_0, 1/q_1+1/q_2=1/q_0$. For a study of modulation spaces in connection with time-frequency analysis, we refer to the books \cite{corodaro20}, \cite{grochenig_book} and for applications to pseudo-differential operators and partial differential equations, we refer to \cite{benyi20}.}

\par In this paper, we look at some fundamental results in Fourier analysis in connection with $A$-translation, $A$-convolution and the SAFT. We will show that the usual approximate identity in $L^1(\R)$ also act in an expected way on $L^r(\R)$ via $A$-convolution. Further, we obtain an analogue of
the fundamental formula for the Fourier transform,
Young's inequality,  and the Hausdorff-Young inequality
for the SAFT. In the classical case, it is well known that $\widehat{(f^\prime)}(\xi)=2\pi i \xi \widehat{f}(\xi)$. We obtain an analogous result for the SAFT using the differential operator $B=\frac{d}{dt}+\frac{2\pi ia }{b}t.$ We also show that the operator $B$ commutes with $A$-translations. We further obtain the solution $u$ of the heat equation associated with the operator $B^*B$ in terms of $A$-convolution, namely $u(x,t)= g\star_Ah_t(x)$, where $u(x,0)=g(x)$, as in the classical case. It is well known that a multiplicative linear functional $h$ on $(L^1(\R),\star)$ is of the form $h(f)=\widehat{f}(\xi_0)$, for some $\xi_0 \in \R$. On the other hand in \cite{Jaming2010}, Jaming proved that if $T:L^1(\R)\to C_0(\R)$ is a continuous linear operator which converts a convolution product into the pointwise product, then $T$ turns out to be $T(f)(\xi)=\bigchi_E(\xi)\widehat{f}(\phi(\xi)),\ \xi \in \R$, for some $E\subset \R$ and a function $\phi:\R \to \R.$ We prove the analogue of both the results in this paper using SAFT.

\par In the next part of the paper we study some fundamental properties of the SAFT in connection with time-frequency analysis. In fact, we write the covariance property of the short time Fourier transform using $A$-translation and $A$-modulation. We also obtain an analogue of fundamental identity of time-frequency analysis using the SAFT. We prove that $f\in \Msp^r_m$ will give $C_sf\in \Msp^{r}_{\nu_s}$, where $$\textcolor{black}{C_sf(t)=e^{\pi i st^2}f(t)}$$ and $\nu_s(x,\omega)=m(x,\omega-sx)$ and if $f\in \Msp_{v_\ell}^r$ then $\F_A(f)\in \Msp_{w_\ell}^r$, where $v_\ell(x,\omega)=(1+x^2+\omega^2)^{\ell/2}$ and $w_\ell=\big(1+(c^2+d^2)x^2+(a^2+b^2)\omega^2-2(ac+bd)x\omega\big)^{\ell/2}$. \textcolor{black}{As a consequence, we conclude that the modulation space $\Msp^r_{v_\ell}$ is invariant under SAFT.} Further, we prove that the special affine Fourier transform and the pointwise multiplication operators defined using the auxiliary functions $\lambda_A, \rho_A$ and $\eta_A$ are Banach Gelfand triple automorphisms on the Banach Gelfand triple $(S_0(\R),L^2(\R),S_0^\prime(\R))$, as described in \cite{cofelu08}. Here $S_0(\R)$ coincides with the modulation space $M^1(\R)$, and the dual space
is just $M^\infty(\R)$ (see \cite{grochenig_book}).

\par We introduce \textcolor{black}{weighted} modulation spaces $\Msp^{r,s}_A$ in connection with the SAFT and study some fundamental properties. \textcolor{black}{It is well known that the classical weighted modulation space $\Msp^1_v$ possesses certain minimality property, namely if a time-frequency shift invariant Banach space $B$ of tempered distributions has non-zero intersection with $\Msp^1_v$, then $\Msp^1_v$ is continuously embedded in $B$. We prove an analogous minimality property for the $A$-modulation space $\Msp^1_{A,v}$ under $A$-time-frequency shifts.} Subsequently we study multipliers in connection with the SAFT and verify how some of the classical theorems in multiplier theory for Fourier transform can be transferred to this setting.
\par In \cite{FeiAcha06}, Feichtinger and Narimani proved that if $T:\Msp^{r_1,s_1}\to \Msp^{r_2,s_2}$ is bounded, linear and satisfying $TT_x=T_xT$ for all $x\in \R$, then there exists a unique $u\in S_0^\prime(\R)$ such that $Tf=u\star f,$ for $f\in S_0(\R).$ In this paper, we prove that if $T:\Msp^{r_1,s_1}_A\to \Msp^{r_2,s_2}_A$ is bounded, linear and satisfying $TT_x^A=T_x^AT$ then there exists a unique $u\in S_0^\prime(\R)$ such that $Tf=u\star_Af$, for all $f\in S_0(\R)$. We also establish H\"{o}rmander multiplier theorem and Littlewood-Paley theorem associated with the SAFT.

\section{Notation and preliminaries}

\par Let $C_c(\R)$ denote the space of all compactly supported continuous,
complex valued functions on $\R$, $S(\R)$, the space of Schwartz class
of rapidly decreasing functions and its dual, $S^\prime(\R)$, the space of tempered distributions. For a complex valued function $f$ defined on $\R$, we write \textcolor{black}{$f^\checkmark (x)=f(-x)$} for the involution. We also
make use of the following standard notations
$$T_sf(t)=f(t-s), \quad M_sf(t)= e^{2\pi ist}f(t), \quad D_sf(t)=\frac{1}{\sqrt{|s|}} f(t/s),
 $$
for the translation, modulation, dilation operators respectively.
A particular role in our  presentation is taken by the
chirp modulation operator, given by $C_\frac{a}{b}f(t)=e^{\frac{\pi ia}{b}t^2}f(t).$

\begin{definition}[\cite{H2}]
For any  $s \in \R$ the $A$-translation of a measurable function $f$ by $s$, denoted by $T_s^Af$, is defined as
\textcolor{black}{\begin{equation}
T_s^Af(t)=e^{-\frac{2\pi i a}{b}s(t-s)}f(t-s).
\end{equation}}
Further, one has
\begin{equation}\label{eq:T_xT_x^A}
C_\frac{a}{b}T_s^Af(t)=e^{\frac{\pi i a}{b}s^2}T_s(C_\frac{a}{b}f)(t).
\end{equation}
\end{definition}

\begin{definition}[\cite{H2}]
For $f,g\in L^1(\R)$  the $A$-convolution of $f$ and $g$,
is given by
\begin{equation}\label{eq:A-convolution}
 (f\star_Ag)(x)=\frac{1}{\sqrt{|b|}}\int_\R f(s) T_s^Ag(x) ds.
\end{equation}
It is compatible with chirp modulation operator as follows:
\begin{equation}\label{eq:A-convolution, convolution}
C_\frac{a}{b}(f\star_Ag)=\frac{1}{\sqrt{|b|}}(C_\frac{a}{b}f\star C_\frac{a}{b}g).
\end{equation}
\end{definition}
\par \textcolor{black}{It is interesting to note that
for $f,g \in C_c(\R)$  \eqref{eq:A-convolution} can be written as a vector-valued integral
$$
\big (\int_\R T_s^Agf(s)ds\big)(x)=
\delta_x\big( \int_\R T_s^A g d\mu_{f(s)}(s) \big),
$$
where $\delta_x$ is the Dirac measure at $x$ (as described in \cite{fe22}). This can be viewed as an integrated action of $f$ in $L^1(\R)$ on $g$.}
\par We use the following auxiliary functions in this paper.
\begin{itemize}
\item $\rho_A(t)=e^{\frac{\pi i}{b}(at^2+2pt)}$.

\item $\eta_A(\omega)=e^{\frac{\pi i}{b}(d\omega^2+2\Omega\omega)}, \ \Omega=bq-dp.$

\item $\lambda_A(t)=e^{\frac{\pi i}{b}at^2}.$\\
\end{itemize}
\par \textcolor{black}{Notice that $\lambda_A, \rho_A,\eta_A$ are continuous functions of absolute value $1$. Hence these auxiliary functions induce unitary pointwise multiplication operators on $L^2(\R)$ and isometry on $L^r(\R)$ for $1\leq r\leq \infty$.}

\par 
They allow to express the SAFT with the help of the classical
Fourier transform  as follows:
\textcolor{black}{\begin{equation}\label{eq:SAFT-FT}
\F_A(f)(\omega)=\frac{\eta_A(\omega)}{\sqrt{|b|}}(\rho_Af)~\widehat{}~(\omega/b)= \frac{\eta_A(\omega)}{\sqrt{|b|}}(C_\frac{a}{b}f)~\widehat{}~\big(\frac{\omega-p}{b}\big).
\end{equation}}

\begin{definition}
\textcolor{black}{For any $s \in \R$ the   $A$-modulation of a measurable function $f$ is defined as:}
\textcolor{black}{\begin{equation}
 M_s^Af(t)=\rho_A(-s)M_\frac{s}{b}f(t)=e^{\frac{\pi i}{b}(as^2-2ps+2st)}f(t).
 \end{equation}}
The modified version of translation, modulation and Fourier transform are
compatible in the expected way, i.e.\ one has
\textcolor{black}{\begin{equation}
\F_A(T_s^Af)(\omega)=M_{-s}^A\F_A(f)(\omega).
\end{equation}}
\end{definition}

\begin{definition}\label{S_S' convolution}\textcolor{black}{\cite{Rudin_FA}}
Let $\phi \in S(\R)$ and $\Lambda \in S^\prime (\R)$. Then $\Lambda \star \phi$ is defined as
$$
(\Lambda \star \phi) (x)=\Lambda(T_x \textcolor{black}{\phi^\checkmark)}, \ \ x\in \R.
$$
\end{definition}

\par Recall that a weight function $\nu(x,\omega)$ on $\R^2$ is said to be \textcolor{black}{\textit{submultiplicative}} if $\nu(x_1+x_2,\omega_1+\omega_2)\leq \nu(x_1,\omega_1)\nu(x_2,\omega_2)$. A weight function $m(x,\omega)$ is said to be $\nu$ \textcolor{black}{\textit{moderate}} if there exists $C>0$ such that $m(x_1+x_2,\omega_1+\omega_2)\leq Cm(x_1,\omega_1)\nu(x_2,\omega_2)$. A weight function $m$ is called \textcolor{black}{\textit{moderate}} if it is moderate with respect to some submultiplicative weight. \textcolor{black}{A weight $m$ is said to have polynomial growth if there exists $C>0$ such that $m(x,\omega)\leq C(1+x+\omega)^s$, for some $s\geq 0$. In this paper, we consider moderate weights of polynomial growth.}
\par Let $m(x,\omega)$ be a moderate weight \textcolor{black}{of polynomial growth} on $\R^2$, $g\in S(\R)$ and $1\leq r,s<\infty$. Then the \textcolor{black}{\textit{modulation space}} is defined as follows.
$$
\Msp_{m}^{r,s}=\{f\in S^\prime (\R):\|f\|_{\Msp_{m}^{r,s}}<\infty\},
$$
where $$\|f\|_{\Msp_{m}^{r,s}}=\left(\int_\R\big(\int_\R |f\star M_\omega g(x)|^rm(x,\omega)^rdx\big)^\frac{s}{r}d\omega\right)^\frac{1}{s}.$$
If $r=s,$ then the modulation space $\Msp^{r,r}_{m}$ is denoted by $\Msp^r_{m}$ and if $m(x,\omega)=1,$ then we write $\Msp^{r,s}$, $\Msp^r$ for $\Msp^{r,s}_{m}$, $\Msp^{r,r}_{m}$ respectively.
\par The modulation space $\Msp^1$ is popularly known as \textcolor{black}{\textit{Feichtinger Segal algebra}}, denoted by $S_0(\R)$ and its dual is denoted by $S_0^\prime(\R).$ \par One can define $\Lambda\star \phi$ for $\phi \in S_0(\R)$ and $\Lambda \in S_0^\prime(\R)$ as in Definition \ref{S_S' convolution}.

\par It is well known that modulation spaces are invariant under translation, modulation and dilation.

\begin{definition}
A linear functional $h$ on a Banach algebra $\B$ is said to be a \textcolor{black}{\textit{multiplicative linear functional}} if $h(xy)=h(x)h(y)$ for all $x,y\in \B.$
\end{definition}

\begin{definition}
A \textcolor{black}{\textit{Banach Gelfand triple}} consists of a Banach space $(\B, \|\cdot \|_\B)$ which is continuously and densely embedded into some Hilbert space $\Hi$, which in turn is $weak^\star$-continuously and densely embedded into the dual Banach space $(\B^\prime, \| \cdot \|_{\B^\prime})$.
\par The well known examples are $\big(S_0(\R), L^2(\R),S_0^\prime(\R)\big)$, \textcolor{black}{$\big( \Hi_s(\R), L^2(\R), \Hi_s^\prime(\R) \big)$, where the Sobolev space $\Hi_s(\R)$ is defined by
$$
\Hi_s(\R)=\{f:(1+|\cdot |^2)^{s/2}\widehat{f}\in L^2(\R)\}.
$$}
\end{definition}

\begin{definition}\cite{Feich_Banach_Gelfand_triple}
If $(\B_1, \Hi_1, \B^\prime_1)$ and $(\B_2, \Hi_2, \B^\prime_2)$ are Banach Gelfand triples then an operator $T$ is called a \textcolor{black}{\textit{[unitary] Gelfand triple isomorphism}} if
\begin{itemize}
\item[(i)] $T$ is an isomorphism between $\B_1$ and $\B_2$.
\item[(ii)] $T$ is a [unitary operator resp.] isomorphism from $\Hi_1$ to $\Hi_2$.
\item[(iii)] $T$ extends to a $weak^\star$ isomorphism as well as a norm-to-norm continuous isomorphism between $\B^\prime_1$ and $\B^\prime_2$.
\end{itemize}
\end{definition}

\begin{lemma}\cite{Feich_Zimmermann} \label{lemma:gelfand automorphism}
Let $T:L^2(\R) \to L^2(\R)$ be a unitary operator. Then the operator $T$ extends to a Banach Gelfand triple isomorphism between $\big(S_0(\R), L^2(\R), S_0^\prime(\R)\big)$ and $\big(S_0(\R), L^2(\R), S_0^\prime(\R)\big)$ if and only if the restriction $T|_{S_0(\R)}$ defines a bounded bijective linear mapping of $S_0(\R)$ onto itself.
\end{lemma}

\section{Fourier analysis and the special affine Fourier transform}
We consider $C_{-\frac{a}{b}}\bigchi_{[0,1]}$ and look at its SAFT. For $\omega \neq p$ we have

\begin{align*}
\F_A(C_{-\frac{a}{b}}\bigchi_{[0,1]})(\omega)=\frac{\eta_A(\omega)}{\sqrt{|b|}}\bigchi_{[0,1]}~\widehat{ }~(\frac{\omega -p}{b})=\frac{\eta_A(\omega)}{\sqrt {|b|}}\frac{1-e^{-2\pi i\frac{\omega -p}{b}}}{2\pi i\frac{\omega-p}{b}},
\end{align*}
and $\F_A(C_{-\frac{a}{b}}\bigchi_{[0,1]})(p)=\frac{\eta_A(p)}{\sqrt{|b|}}$.

Similarly we have $\F_A(C_{-\frac{a}{b}}\bigchi_{[-\frac{1}{2},\frac{1}{2}]})(\omega)=\frac{\eta_A(\omega)}{\sqrt{|b|}}sinc( \frac{w-p}{b}),$ where
\begin{equation*}
sinc(x) = \begin{cases}
            \frac{sin\pi x}{\pi x}, &x\neq0\\
            1, & x=0
            \end{cases}.
\end{equation*}

In particular, for the fractional Fourier transform, we take $A_\theta=\{cos\theta, sin\theta,\\ -sin\theta, cos\theta, 0 ,0\}$. Then
$$
\F_{A_\theta}(C_{-\frac{a}{b}}\bigchi_{[-\frac{1}{2},\frac{1}{2}]})(\omega)=\frac{1}{\sqrt{|sin\theta|}}e^{\pi i \omega^2cot\theta}sinc(\omega/sin\theta).
$$

\textcolor{black}{\begin{proposition}
The special affine Fourier transform $\F_A$ is a bijection on $S(\R)$.
\end{proposition}}
\begin{proof}
We have
\begin{align*}
\F_A(f)(\omega)=\frac{\eta_A(\omega)}{\sqrt{|b|}}\F(\rho_Af)(\omega/b)&=e^{\frac{\pi i}{b}(d\omega^2+2(bq-dp)\omega)}D_b\F(\rho_Af)(\omega)\\
&= C_\frac{d}{b}M_\frac{bq-dp}{b}D_b\F(C_\frac{a}{b}M_\frac{p}{b}f)(\omega),
\end{align*}
where $\F$ denotes the classical Fourier transform. In other words, $\F_A$ is a composition of $C_s$, modulation, dilation and the classical Fourier transform. It is easy to see that $C_s$, modulation and dilation are bijections on $S(\R)$. Hence the result follows from the fact that $\F:S(\R)\to S(\R)$ is a bijection.
\end{proof}

\begin{theorem}
Let $f,g\in S(\R)$. Then we have the
\textcolor{black} {
variants of the fundamental relation of Fourier analysis} for the SAFT.
\begin{itemize}
\item[(i)] $\int_\R \overline{\eta}_A(\omega)\F_A(\overline{
\rho}_Af)(\omega)g(\omega)d\omega = \int_\R \overline{\eta}_A(\omega)f(\omega)
\F_A(\overline{\rho}_Ag)(\omega)d\omega.$
\item[(ii)] $\int_\R \overline{\eta}_A(\omega)\F_Af (\omega)\rho_Ag(\omega)d\omega = \int_\R \overline{\eta}_A(\omega)\rho_Af(\omega)\F_Ag
(\omega)d\omega. $
\item[(iii)]\label{multiplication formula} If $a=d$ and $p=q=0$, then $\int_\R \F_A(f)(\omega)g(\omega)d\omega =\int_\R f(\omega)\F_A(g)(\omega)d\omega.$
\end{itemize}
\end{theorem}
\begin{proof}

We have
\begin{align*}
\F_Af(\omega)=\frac{\eta_A(\omega)}{\sqrt{|b|}}(\rho_Af)~\widehat{}~(\omega/b)
\end{align*}
and
\begin{align*}
\F_A(\overline{\rho}_Af)(\omega)=\frac{\eta_A(\omega)}{\sqrt{|b|}}\widehat{f}(\omega/b).
\end{align*}
(i) Consider
\begin{align*}
\int_\R \overline{\eta}_A(\omega)\F_A(\overline{
\rho}_Af)(\omega)g(\omega)d\omega &= \frac{1}{\sqrt{ |b|}}\int_\R \widehat{f}(\omega/b)g(\omega)d\omega\\
&=\int_\R D_b\widehat{f}(\omega)g(\omega)d\omega\\
&=\int_\R (D_{1/b}f)\ \widehat{} \ (\omega)g(\omega)d\omega\\
&=\int_\R D_{1/b}f(\omega)\widehat{g}(\omega)d\omega\\
&=\sqrt{|b|}\int_\R f(b\omega)\widehat{g}(\omega)d\omega\\
&= \frac{1}{\sqrt{ |b|}} \int_\R f(\omega)\widehat{g}(\omega/b)d\omega\\
&=\int_\R \overline{\eta}_A(\omega)f(\omega)
\F_A(\overline{\rho}_Ag)(\omega)d\omega,
\end{align*}
using 
the fundamental formula for the Fourier transform and applying a change of variables.\\
Similarly we can prove (ii).\\
(iii) Consider
\begin{align*}
\int_\R \F_A(f)(\omega)g(\omega)d\omega &=\frac{1}{\sqrt{|b|}}\int_\R g(\omega) \int_\R e^{\frac{\pi i}{b}(at^2+a\omega^2 -2\omega t)}f(t)dtd\omega\\
&=\int_\R f(t)\int_\R \frac{1}{\sqrt{|b|}}e^{\frac{\pi i}{b}(at^2+a\omega^2 -2\omega t)}g(\omega)d\omega dt\\
&=\int_\R f(t)\F_A(g)(t)dt,
\end{align*}
using Fubini's theorem.
\end{proof}

\textcolor{black}{\begin{remark}
We can take $f,g\in S_0(\R)$ in the 
fundamental formula, either by observing that all the
integrals are well defined (even as Riemann integrals).
\end{remark}}

\begin{remark}
Theorem \ref{multiplication formula} (iii) leads to the 
fundamental formula for the fractional Fourier transform and the Fresnel transform. Let $A_\theta=\{cos\theta, sin\theta, -sin\theta, cos\theta, 0 ,0\}$ and $A_\lambda=\{1,\lambda,0,1,0,0\}$. Then the fundamental formula
for the fractional Fourier transform and the Fresnel transform read
\begin{equation}
\int_\R \F_{A_\theta}(f)(\omega)g(\omega)d\omega =\int_\R f(\omega)\F_{A_\theta}(g)(\omega)d\omega,
\end{equation}
and
\begin{equation}
\int_\R \F_{A_\lambda}(f)(\omega)g(\omega)d\omega =\int_\R f(\omega)\F_{A_\lambda}(g)(\omega)d\omega
\end{equation}
respectively.
\end{remark}

\par Let $B$ be the differential operator defined by $\frac{d}{dt}+\frac{2\pi i a}{b}t$. Then $B^*=-B$. Further $BB^*=B^*B= -(\frac{d^2}{dt^2}+\frac{4\pi i a}{b}t\frac{d}{dt}-\frac{4\pi ^2a^2}{b^2}t^2+\frac{2\pi i a}{b}I),$ where $I$ is the identity operator.

\begin{proposition}\label{pro:diff operator}
If $f\in S(\R)$, then the following statements hold.
\begin{itemize}
\item[(i)] $\F_A(Bf)(\omega)=2\pi i\frac{\omega -p}{b}\F_A(f)(\omega).$
\item[(ii)] $B(\F_Af)(\omega)=\frac{2\pi i}{b}(a\omega +d\omega +\Omega)\F_A(f)(\omega)-\frac{2\pi i}{b}\F_A(tf(t))(\omega),$\\ where $\Omega=bq-dp.$
\end{itemize}
\end{proposition}
\begin{proof}
(i) Consider
\begin{align*}
\F_A(f^\prime)(\omega)&=\frac{\eta_A(\omega)}{\sqrt{|b|}}\int_\R e^{\frac{\pi i}{b}(at^2+2pt-2\omega t)}f^\prime(t) dt\\
&=\frac{\eta_A(\omega)}{\sqrt{|b|}}\int_\R e^{-2\pi i\frac{\omega-p}{b}t}C_\frac{a}{b}(f^\prime)(t)dt\\
&= \frac{\eta_A(\omega)}{\sqrt{|b|}}\int_\R e^{-2\pi i\frac{\omega-p}{b}t}\big((C_\frac{a}{b}f)^\prime (t)-\frac{2\pi ia}{b}C_\frac{a}{b}(tf)(t)\big)dt\\
&=\frac{\eta_A(\omega)}{\sqrt{|b|}}\big((C_\frac{a}{b}f)^\prime \big)~\widehat{}~(\frac{\omega-p}{b})-\frac{2\pi ia}{b}\frac{\eta_A(\omega)}{\sqrt{|b|}}\int_\R e^{-2\pi i\frac{\omega -p}{b}t}C_\frac{a}{b}(tf)(t)dt\\
&=\frac{\eta_A(\omega)}{\sqrt{|b|}}2\pi i\frac{\omega -p}{b}(C_\frac{a}{b}f)~\widehat{ }~(\frac{\omega -p}{b})-\frac{2\pi ia}{b}\frac{\eta_A(\omega)}{\sqrt{|b|}}\\
&\times \int_\R e^{\frac{\pi i}{b}(at^2+2pt-2\omega t)}tf(t)dt\\
&= 2\pi i\frac{\omega -p}{b}\F_A(f)(\omega)-\frac{2\pi ia}{b}\F_A(t  f(t))(\omega),
\end{align*}
using $C_\frac{a}{b}(f^\prime)(t)=(C_\frac{a}{b}f)^\prime (t)-\frac{2\pi i a}{b}C_\frac{a}{b}(tf(t))(t),$ where $f^\prime=\frac{df}{dt}.$
Thus, $\F_A(Bf)(\omega)=2\pi i \frac{\omega -p}{b}\F_A(f)(\omega).$

(ii) Consider
\begin{align*}
\frac{d}{d\omega}(\F_Af)(\omega)&=\frac{1}{\sqrt{|b|}}\frac{d\eta_A}{d\omega}(\omega)\int_\R e^{\frac{\pi i}{b}(at^2+2pt-\omega t)}f(t)dt\\
&+ \frac{\eta_A(\omega)}{\sqrt{|b|}}\frac{d}{d\omega}\int_\R e^{\frac{\pi i}{b}(at^2+2pt-2\omega t)}f(t)dt\\
&= \frac{2\pi i}{b}(d\omega +\Omega)\frac{\eta_A(\omega)}{\sqrt{|b|}}\int_\R e^{\frac{\pi i}{b}(at^2+2pt-2\omega t)}f(t)dt\\
&+ \frac{\eta_A(\omega)}{\sqrt{|b|}}\frac{d}{d\omega}((C_\frac{a}{b}f) \ \widehat{} \ (\frac{\omega-p}{b}))\\
&=\frac{2\pi i}{b}(d\omega +\Omega)\F_A(f)(\omega)-\frac{2\pi i}{b}\frac{\eta_A(\omega)}{\sqrt{|b|}}(C_\frac{a}{b}(tf(t)))~\widehat{}~(\frac{\omega-p}{b})\\
&=\frac{2\pi i}{b}(d\omega +\Omega)\F_A(f)(\omega)-\frac{2\pi i}{b}\F_A(tf(t))(\omega).
\end{align*}
Thus $B(\F_Af)(\omega)=\frac{2\pi i}{b}(a\omega +d\omega +\Omega)\F_A(f)(\omega)-\frac{2\pi i}{b}\F_A(tf(t))(\omega).$
\end{proof}

\par It is interesting to note that $B$ commutes with $A$-translations. In fact,

\begin{align*}
\F_A(BT_x^Af)(\omega)&=2\pi i\frac{\omega -p}{b}\F_A(T_x^Af)(\omega)\\
&=2\pi i\frac{\omega -p}{b}e^{\frac{\pi i}{b}(ax^2+2px-2xt)}\F_A(f)(\omega)\\
&=e^{\frac{\pi i}{b}(ax^2+2px-2xt)}\F_A(Bf)(\omega)\\
&=\F_A(T_x^ABf)(\omega).
\end{align*}
Then it follows from uniqueness of the SAFT that  $BT_x^A=T_x^AB.$
\par Consider the heat equation associated with the operator $B^*B$, given by,
\begin{equation}\label{eq:heat}
\frac{\partial u}{\partial t}(x,t)=-B^*Bu(x,t),
\end{equation}
with initial condition $u(x,0)=g(x),\ x\in \R \ \text{and} \ t>0 $. We shall obtain the solution of (\ref{eq:heat}).\\
By (\ref{eq:heat}),
$$
\frac{\partial}{\partial t}\F_Au(\omega,t)=-\big(2\pi \frac{\omega -p}{b}\big)^2\F_Au(\omega, t),
$$
using Proposition \ref{pro:diff operator} (i).
Thus
$$
\F_Au(\omega,t)=\F_Au(\omega,0)e^{-(2\pi\frac{\omega -p}{b})^2t}=\F_A(g)(\omega)e^{-(2\pi\frac{\omega -p}{b})^2t}.
$$
Let $g_t(x)=e^{-\frac{x^2}{4t}}$. Then
$$
\F_A(C_{-\frac{a}{b}}g_t)(\omega)=\frac{\eta_A(\omega)}{\sqrt{|b|}}\widehat{g_t}(\frac{\omega -p}{b})=\eta_A(\omega)\sqrt{\frac{4\pi t}{|b|}}e^{-4\pi^2t(\frac{\omega -p}{b})^2}.
$$
Hence
$$
\F_Au(\omega,t)=\sqrt{\frac{|b|}{4\pi t}}\overline{\eta}_A(\omega)\F_A(g)(\omega)\F_A(C_{-\frac{a}{b}}g_t)(\omega)=\sqrt{\frac{|b|}{4\pi t}}\F_A(g\star_AC_{-\frac{a}{b}}g_t)(\omega).
$$
It follows that $$u(x,t)=\sqrt{\frac{|b|}{4\pi t}}(h_t\star_Ag)(x),$$
where $h_t(x)=C_{-\frac{a}{b}}g_t(x)$ given in terms of $A$-convolution. On the other hand,
\begin{align*}
(C_{-\frac{a}{b}}g_t\star_Ag)(x)&=\frac{1}{\sqrt{|b|}}\int_\R T_y^A(C_{-\frac{a}{b}}g_t)(x)g(y)dy\\
&=\frac{1}{\sqrt{|b|}}\int_\R e^{-\frac{2\pi i a}{b}y(x-y)}e^{-\frac{\pi ia}{b}(x-y)^2}g_t(x-y)g(y)dy.
\end{align*}
Upon simplification we get
$$
u(x,t)=\frac{1}{\sqrt{4\pi t}}\int_\R e^{-\frac{\pi ia}{b}(x^2-y^2)}e^{-\frac{1}{4t}(x-y)^2}g(y)dy.
$$

\par Next we describe the action of the usual (bounded) approximate
identities from $L^1(\R)$ on   $L^r(\R)$ through $A$-convolution.

\begin{theorem}\label{approx identity} Given  $r$ with $1 \leq r < \infty$ and
 $\phi \in L^1(\R)$ with  $\int_\R \phi(x)dx = 1$, then for $\phi_\epsilon(x)=\frac{1}{\epsilon}\phi(x/\epsilon)$ one has
 $$\lim_{\epsilon\to 0}\|f\star_A\phi_\epsilon - f\|_r=0, \quad f \in L^r(\R).$$
\end{theorem}
In order to prove this theorem, we observe the following

\begin{proposition}\label{A-translation conts}
For $f\in L^r(\R)$, $1\leq r <\infty$ one has $\|T_h^Af-f\|_r\to 0$ as $h\to 0.$
\end{proposition}
\begin{proof}
Consider
\begin{align*}
\|T_h^Af-f\|_r^r&=\int_\R |e^{-\frac{2\pi ia}{b}h(t-h)}f(t-h)-f(t)|^rdt\\
&\leq \int_\R |e^{-\frac{2\pi ia}{b}h(t-h)}f(t-h)-e^{-\frac{2\pi ia}{b}h(t-h)}f(t)|^rdt\\
&+\int_\R |f(t)|^r|e^{-\frac{2\pi ia}{b}h(t-h)}-1|^rdt\\
&=\|T_hf-f\|_r^r+\int_\R |f(t)|^r|e^{-\frac{2\pi ia}{b}h(t-h)}-1|^rdt.
\end{align*}
Since $f\in L^r(\R)$, using dominated convergence theorem, we can show that the second term on the right hand side tends to $0$ as $h\to 0$. Thus the result follows from the fact that $\|T_hf-f\|_r\to 0$ as $h\to 0$.
\end{proof}

\begin{proof}[Proof of Theorem \ref{approx identity}]
Consider
\begin{align*}
\|f\star_A\phi_\epsilon -f\|_r&=\bigg(\int_\R |(f\star_A\phi_\epsilon)(x)-f(x)|^rdx\bigg)^{1/r}\\
&=\bigg(\int_\R|\int_\R T_t^Af(x)\phi_\epsilon(t)dt-f(x)|^rdx\bigg)^{1/r}\\
&= \bigg(\int_\R |\int_\R (T_t^Af(x)-f(x))\phi_\epsilon (t)dt|^rdx\bigg)^{1/r}\\
&\leq \int_\R |\phi_\epsilon (t)| \ \|T_t^Af-f\|_rdt\\
&=\int_\R |\phi(t)| \ \|T_{\epsilon t}^Af-f\|_rdt,
\end{align*}
using Minkowski's integral inequality. Now an application of Proposition \ref{A-translation conts} and Lebesgue dominated convergence theorem gives $\displaystyle \lim_{\epsilon \to 0}\|f\star_A\phi_\epsilon -f\|_r =0.$
\end{proof}

\textcolor{black}{\begin{theorem}(Riemann-Lebesgue lemma)
If $f\in L^1(\R)$ then $\F_A(f)\in C_0(\R)$.
\end{theorem}
\begin{proof}
We know that
$$
\F_A(f)(\omega)=\frac{\eta_A(\omega)}{\sqrt{|b|}}(\rho_Af)~ \widehat~ (\omega/b).
$$
Since $f\in L^1(\R)$, $\rho_Af\in L^1(\R)$. Applying classical Riemann-Lebesgue lemma, we obtain $\omega \mapsto (\rho_Af)~\widehat{}~(\omega)$ is continuous. This implies that $\omega\mapsto (\rho_Af)~\widehat{}~(\omega/b)$ is continuous. Further, $\omega \mapsto \eta_A(\omega)$ is continuous, from which it follows that $\F_A(f)$ is continuous. Moreover
$$
|\F_A(f)(\omega)|=\frac{1}{\sqrt{|b|}}|(\rho_Af)~\widehat{}~(\omega/b)|
$$
and it follows from classical Riemann-Lebesgue lemma that $\F_A(f)(\omega)\to  0$ as $|\omega|\to \infty$.
\end{proof}}

\begin{theorem}(Hausdorff-Young)
Let $1\leq r\leq 2$ and $1/r+1/r^\prime=1$. Then $\F_A:L^r(\R)\to L^{r^\prime}(\R)$ is a bounded linear operator \textcolor{black}{with $\|\F_A\|_{L^r\to L^{r^\prime}}\leq \frac{1}{\sqrt{|b|}}\|\F\|^{2/r-1}_{L^1\to L^\infty}$, where $\F$ is the classical Fourier transform.}
\end{theorem}
\textcolor{black}{\begin{proof}
Consider
\begin{align*}
\|\F_A\|_{L^1\to L^\infty}= \sup_{\|f\|_1=1}\sup_{\omega\in \R}|\F_A(f)(\omega)|&=\frac{1}{\sqrt{|b|}}\sup_{\|f\|_1=1}\sup_{\omega\in \R} |\F(\rho_Af)(\omega/b)|\\
&= \frac{1}{\sqrt{|b|}} \sup_{\|\rho_Af\|_1=1} \sup_{\omega \in \R}|\F(\rho_Af)(\omega)|\\
&= \frac{1}{\sqrt{|b|}}\|\F\|_{L^1\to L^\infty},
\end{align*}
using \eqref{eq:SAFT-FT}. Further $\|\F_A\|_{L^2\to L^2}=1$. Now applying Riesz-Thorin convexity theorem we get $\F_A:L^r \to L^{r^\prime}$ is bounded for $1\leq r \leq 2$ and $\|\F_A\|_{L^r\to L^{r^\prime}}\leq \|\F_A\|^t_{L^1\to L^\infty}$, where $t$ is given by $1/r=1/2+t/2$. In other words, $\|\F_A\|_{L^r\to L^{r^\prime}}\leq \frac{1}{\sqrt{|b|}}\|\F\|^{2/r-1}_{L^1\to L^\infty}$ for $1\leq r\leq 2$.
\end{proof}}

\begin{remark}
The proof will show that even the more detailed behaviour of the Fourier 
transform as expressed by the corresponding theorem in \cite{ho75} can
be derived. A proof, which in some sense is closer to the spirit of 
Wiener amalgam spaces, can be derived from  Theorem 3.2 of \cite{fe81-1}. 
It implies among others that the SAFT maps
$W(L^p,l^q)(\R)$ boundedly into $W(L^{q'},l^{p'})$, for $1 \leq p,q \leq 2$. 
In this sense local properties of $f$ imply global properties of the transform
and vice versa.  
\end{remark} 

\begin{theorem}(Young)
If $f\in L^r(\R), \ g\in L^s(\R)$ and $1/r+1/s=1+1/t$ for $1\leq r,s,t\leq \infty$, then $$\textcolor{black}{f\star_Ag\in L^t(\R).}$$
\end{theorem}
\begin{proof}
From \eqref{eq:A-convolution, convolution} we have $C_\frac{a}{b}(f\star_Ag)=\frac{1}{\sqrt{|b|}}(C_\frac{a}{b}f\star C_\frac{a}{b}g)$. This leads to $$|f\star_Ag|=\frac{1}{\sqrt{|b|}}|C_\frac{a}{b}f\star C_\frac{a}{b}g|.$$ Hence,
\begin{align*}
\|f\star_Ag\|_t&=\frac{1}{\sqrt{|b|}}\|C_\frac{a}{b}f\star C_\frac{a}{b}g\|_t \leq \frac{1}{\sqrt{|b|}}\|C_\frac{a}{b}f\|_r\|C_\frac{a}{b}g\|_s\leq \frac{1}{\sqrt{|b|}} \|f\|_r \|g\|_s,
\end{align*}
using classical Young's inequality and the fact that $C_\frac{a}{b}$ is an isometry on $L^r(\R)$.
\end{proof}


\begin{theorem}
(i) Any  multiplicative linear functional $h$ on $(L^1(\R),\star_A)$ is
of the form
$$
h(f)=\overline{\eta}_A(\omega_0)\F_A(f)(\omega_0), \ f\in L^1(\R),
$$
for some $\omega_0 \in \R$.\\
\textcolor{black}{(ii) Let $T:L^1(\R)\to C_0(\R)$ be a continuous linear operator satisfying
$$T(f\star_Ag)(\omega)=\overline{\eta}_A(\omega)T(f)(\omega)T(g)(\omega)$$
for all $f,g\in L^1(\R)$. Then there exist $E\subset \R$ and $\phi:\R\to \R$ such that
$$T(f)(\omega)=\bigchi_E(\omega)\F_A(f)(\phi(\omega)), \quad f \in L^1(\R).
$$}
\end{theorem}
\begin{proof}
(i) Consider
\begin{align*}
h(f)h(g)=h(f\star_Ag)=h(C_\frac{a}{b}^{-1}C_\frac{a}{b}(f\star_Ag))=\frac{1}{\sqrt{|b|}}h\big(C_\frac{a}{b}^{-1}(C_\frac{a}{b}f\star C_\frac{a}{b}g)\big),
\end{align*}
using (\ref{eq:A-convolution, convolution}). This in turn leads to
$$
h(C_\frac{a}{b}^{-1}f)h(C_\frac{a}{b}^{-1}g)=\frac{1}{\sqrt{|b|}}h\big(C_\frac{a}{b}^{-1}(f\star g)\big).
$$
Let $\tilde{h}(f)=\sqrt{|b|}h(C_\frac{a}{b}^{-1}f)$. Then $\tilde{h}(f\star g)=\tilde{h}(f)\tilde{h}(g)$. In other words, $\tilde{h}$ is a multiplicative linear functional on $(L^1(\R),\star)$. Hence there exists $\omega_0\in \R$ such that $\tilde{h}(f)=\widehat{f}\big(\frac{\omega_0-p}{b}\big)$. (See Chapter VIII, Theorem 2.10 in \cite{katznelson}). Therefore
$$
\sqrt{|b|}h(C_\frac{a}{b}^{-1}f)=\widehat{f}\big(\frac{\omega_0-p}{b}\big),
$$
which  implies the final assertion:
$$
h(f)=\frac{1}{\sqrt{|b|}}(C_\frac{a}{b}f)~\widehat{}~\big(\frac{\omega_0-p}{b}\big).
$$

(ii) As in (i), we can show that
$$
\frac{1}{\sqrt{|b|}} T(C_\frac{a}{b}^{-1}(f\star g))(\omega)=\overline{\eta}_A(\omega)T(C_\frac{a}{b}^{-1}f)(\omega)T(C_\frac{a}{b}^{-1}g)(\omega).
$$
Expressed differently we have for $f,g \in L^1(\R)$:
$$
\sqrt{|b|}\overline{\eta}_A(\omega)T(C_\frac{a}{b}^{-1}(f\star g))(\omega)=\sqrt{|b|}\overline{\eta}_A(\omega)T(C_\frac{a}{b}^{-1}f)(\omega)\sqrt{|b|}\overline{\eta}_A(\omega)T(C_\frac{a}{b}^{-1}g)(\omega).
$$
Let $\tilde{T}(f)(\omega)=\sqrt{|b|}\overline{\eta}_A(\omega)T(C_\frac{a}{b}^{-1} f)(\omega)$, or $\tilde{T}(f\star g)=\tilde{T}(f)\tilde{T}(g)$. It then follows from Theorem 3.1 in \cite{Jaming2010}, that there exist $E\subset \R$ and $\phi:\R \to \R$ such that $\tilde{T}(f)(\omega)=\bigchi_E(\omega)\widehat{f}(\frac{\phi(\omega)-p}{b})$. Hence
$$
T(f)(\omega)=\bigchi_E(\omega)\frac{\eta_A(\omega)}{\sqrt{|b|}}(C_\frac{a}{b}f)~\widehat{}~(\frac{\phi(\omega)-p}{b}),
$$
proving our assertion.
\end{proof}

\par Now we shall define $A$-convolution of a tempered distribution and a Schwartz class function and establish a relation with corresponding classical convolution. In order to do so, first, we extend the definition of $C_s$ to the space of tempered distributions.

\begin{definition}
Let $\Lambda\in S^\prime (\R)$. Then $C_s$ is defined on $S^\prime (\R)$ as
$
C_s\Lambda(\phi)=\Lambda(C_s \phi), \ \phi \in S(\R).
$
\end{definition}
\par Observe that $C_s\Lambda \in S^\prime (\R)$, whenever $\Lambda \in S^\prime (\R).$ In order to define $A$-convolution between a tempered distribution and a Schwartz class function, we observe the following. For $f\in L^r(\R)$ and $\phi \in S(\R)$, consider
\begin{align*}
(f\star_A \phi )(x)&=\frac{1}{\sqrt{|b|}}\int_\R f(y)T_y^A\phi(x)dy\\
&= \frac{1}{\sqrt{|b|}}\int_\R f(y)e^{\frac{2\pi ia}{b}y(y-x)}\textcolor{black}{\phi^\checkmark }(y-x)dy\\
&=\frac{1}{\sqrt{|b|}}\int_\R f(y)e^{\frac{2\pi ia}{b}y^2}e^{-\frac{2\pi ia}{b}x^2}e^{-\frac{2\pi ia}{b}x(y-x)}\textcolor{black}{\phi^\checkmark }(y-x)dy\\
&=\frac{e^{-\frac{2\pi ia}{b}x^2}}{\sqrt{|b|}}\int_\R f(y)e^{\frac{2\pi ia}{b}y^2}T_x^A\textcolor{black}{\phi^\checkmark }(y)dy\\
&=\frac{e^{-\frac{2\pi ia}{b}x^2}}{\sqrt{|b|}}\int_\R f(y)C_\frac{2a}{b}T_x^A\textcolor{black}{\phi^\checkmark}  (y)dy.
\end{align*}

\par Thus we give the following
\begin{definition}
For $\phi \in S(\R)$ and $\Lambda \in S^\prime (\R)$, $\Lambda \star_A \phi$ is defined as
$$
(\Lambda \star_A \phi)(x)=  {e^{-\frac{2\pi ia}{b}x^2}}\frac{1}{\sqrt{|b|}} \Lambda (C_\frac{2a}{b}T_x^A\textcolor{black}{\phi^\checkmark}).
$$
\end{definition}

\par Now we establish a relation between $A$-convolution and the corresponding classical convolution of a Schwartz class function with a tempered distribution.

\begin{proposition}\label{Pro:tempered chirp}
For $\Lambda \in S^\prime (\R)$ we get
\begin{equation}
C_\frac{a}{b}(\Lambda \star_A \phi)=\frac{1}{\sqrt{|b|}}(C_\frac{a}{b}\Lambda \star C_\frac{a}{b} \phi), \quad \phi \in S(\R).
\end{equation}
\end{proposition}
\begin{proof}
Consider
\begin{align*}
(\Lambda \star_A \phi )(x)= \frac{e^{-\frac{2\pi ia}{b}x^2}}{\sqrt{|b|}} \Lambda (C_\frac{2a}{b}T_x^A\textcolor{black}{\phi^\checkmark})&= \frac{e^{-\frac{2\pi ia}{b}x^2}}{\sqrt{|b|}}\Lambda (C_\frac{a}{b}C_\frac{a}{b}T_x^A\textcolor{black}{\phi^\checkmark}) \\
&=\frac{e^{-\frac{2\pi ia}{b}x^2}}{\sqrt{|b|}} C_\frac{a}{b}\Lambda(e^{\frac{\pi i}{b}ax^2}T_xC_\frac{a}{b}\textcolor{black}{\phi^\checkmark})\\
&= \frac{e^{-\frac{\pi ia}{b}x^2}}{\sqrt{|b|}}(C_\frac{a}{b}\Lambda \star C_\frac{a}{b}\phi)(x),
\end{align*}
using (\ref{eq:T_xT_x^A}) and $C_\frac{a}{b}(\textcolor{black}{\phi^\checkmark})=(C_\frac{a}{b}\textcolor{black}{\phi)^\checkmark}.$ Hence the result follows.
\end{proof}

\begin{remark}\label{rk:S_0 chirp}
Similar to the arguments above, one can extend the definition of $C_s$ on $S_0^\prime(\R)$ and  $A$-convolution of $\phi \in S_0(\R), \ \Lambda\in S_0^\prime (\R)$. Hence an analogue of Proposition \ref{Pro:tempered chirp} can be obtained by replacing $S(\R)$ with $S_0(\R)$.
\end{remark}

\section{Time-frequency analysis and the SAFT}
Recall that a modulation space can also be defined using short time Fourier transform. The short time Fourier transform of $f$ with respect to a window $g$ is defined by
$$
V_gf(x,\omega)=\int_\R f(t)\overline{g(t-x)}e^{-2\pi i\omega t}dt, \  \ (x,\omega)\in \R^2.
$$
Using this, the modulation space $\Msp^{r,s}_m$ is defined to be $\{f:V_gf\in L^{r,s}_m\}$ for a moderate weight $m$. Here the space $L^{r,s}_m$ is defined to be $\{f:\big(\int_\R (\int_\R |f(x,y)|^rm(x,y)^rdx)^\frac{s}{r}\\dy\big)^{1/s}<\infty\}$.
\par Now we intend to look at the short time Fourier transform of the chirp modulation of a function, from which we show that whenever $f\in \Msp^r_m,$ $C_sf\in \Msp^r_{\nu_s},$ where $\nu_s(x,\omega)=m(x,\omega -sx)$.

\begin{proposition}\label{Chirp_STFT}
\textcolor{black}{(i) If $f\in S_0(\R), g\in S_0^\prime (\R)$, then $$V_{(C_sg)}C_sf(x,\omega)=e^{-\pi i sx^2}V_gf(x,\omega-sx).$$}\\
(ii) $f\in \Msp_m^{r}$ if and only if $C_sf\in \Msp_{\nu_s}^{r},$ where $\nu_s(x,\omega)=m(x,\omega-sx).$
\end{proposition}

\begin{proof}
(i) \textcolor{black}{ Writing short time Fourier transform using the duality relation between $S_0(\R)$ and $S_0^\prime(\R)$ we get
\begin{align*}
V_{(C_sg)}C_sf(x,\omega)=\langle C_sf, M_\omega T_xC_sg\rangle =\langle f, C_{-s}M_\omega T_x C_sg\rangle &=e^{-\pi i sx^2}\langle f, M_{\omega-sx} T_xg\rangle\\
&=e^{-\pi i s x^2}V_gf(x,\omega-sx).
\end{align*} }

(ii) Using (i) we get
\begin{align*}
\|f\|_{\Msp_m^r}^r=\|V_gf\|_{L_m^r}^r&= \int_\R \int_\R |V_gf(x,\omega)|^r m(x,\omega)^rdxd\omega\\
&=\int_\R \int_\R |V_{(C_sg)}C_sf(x,\omega +sx)|^rm(x,\omega)^rdxd\omega\\
&=\int_\R \int_\R |V_{(C_sg)}C_sf(x,\omega)|^rm(x,\omega -sx)^rdxd\omega.
\end{align*}
But there exist $c_1,c_2>0$ such that
\begin{align*}
&c_1\big(\int_\R \int_\R |V_{(C_sg)}C_sf(x,\omega)|^r\nu_s(x,\omega)^rdxd\omega\big)^{1/r}\\
&\leq \big(\int_\R \int_\R |V_gC_sf(x,\omega)|^r\nu_s(x,\omega)^rdxd\omega\big)^{1/r}\\
&\leq c_2\big(\int_\R \int_\R |V_{(C_sg)}C_sf(x,\omega)|^r\nu_s(x,\omega)^rdxd\omega\big)^{1/r}.
\end{align*}

Hence the result follows.
\end{proof}

\begin{remark}
If $m(x,\omega)=1,$ then $\Msp^r$ is invariant under the operator $C_s$.
\end{remark}

\begin{remark}\label{rk:chirp invariance}
The operator $C_s:\Msp^r\to \Msp^r$ is bounded \textcolor{black}{uniformly for $s\in \R$} and bijective. \textcolor{black}{In particular, by choosing
$r =1$,  $C_s$ is a bounded and bijective operator on $S_0(\R)$. In fact, the later case was proved by Feichtinger already in \cite{Feich81}.}
\end{remark}

\textcolor{black}{\begin{proposition}
Let $m$ be a $v$-moderate weight. Then the modulation space $\Msp^{r,s}_m$ is invariant under $A$-time-frequency shifts and
$$
\|M_\omega^AT_x^Af\|_{\Msp_m^{r,s}}\leq v(x,\frac{\omega-ax}{b})\|f\|_{\Msp^{r,s}_m}
$$
\end{proposition}
\begin{proof}
Consider
\begin{align*}
\|M_\omega^AT_x^Af\|_{\Msp^{r,s}_m}&=\|\rho_A(-\omega)M_\frac{\omega}{b}T_xM_{\frac{-ax}{b}}f\|_{\Msp^{r,s}_m}\\
&=\|e^{\frac{2\pi i\omega x}{b}}T_xM_\frac{\omega}{b}M_{-\frac{ax}{b}}f\|_{\Msp^{r,s}_m}\\
&=\|T_xM_{\frac{\omega-ax}{b}}f\|_{\Msp^{r,s}_m}\\
&\leq v(x,\frac{\omega-ax}{b})\|f\|_{\Msp^{r,s}_m},
\end{align*}
using $\|T_xM_\omega f\|_{\Msp^{r,s}_m}\leq v(x,\omega)\|f\|_{\Msp^{r,s}_m}$. (See Theorem 11.3.5 in \cite{grochenig_book}).
\end{proof}}

\par The following result states that the SAFT maps $\Msp^r$ into $\Msp^{r^\prime}$ for $1\leq r \leq 2.$

\begin{theorem}
We have $\F_A(\Msp^{r})\subseteq \Msp^{r^\prime}$ for $1\leq r \leq 2$ and $1/r+1/r^\prime=1$.
\end{theorem}
\begin{proof}
We have
$$
\F_A(f)(\omega)= \frac{\eta_A(\omega)}{\sqrt{|b|}}(C_\frac{a}{b}f)~\widehat{ }~(\frac{\omega-p}{b})= \eta_A(\omega)T_pD_b(C_\frac{a}{b}f)~\widehat{}~(\omega).
$$
Let $f\in \Msp^{r}$. Then $(C_{\frac{a}{b}}f)\in \Msp^{r}$, which in turn implies that $(C_\frac{a}{b}f)~\widehat{}\in \Msp^{r^\prime}$. But one can write $\eta_A(\omega)T_pD_b(C_\frac{a}{b}f)~\widehat{}~(\omega)=C_{d/b}M_{q-dp/b}T_pD_b(C_\frac{a}{b}f)~\widehat{}~(\omega)$. Since the modulation space $\Msp^{r^\prime}$ is invariant under translation, modulation, dilation and chirp modulation, it follows that $\F_A(f)\in \Msp^{r^\prime}$.
\end{proof}

\par We write the covariance property for short time Fourier transform using $A$-translation $T_\xi^A$ and $A$-modulation $M_\eta^A.$
\begin{proposition}
We have
\begin{align*}
V_g(T_\xi^AM_\eta^Af)(x, \omega) = \rho_A(-\eta)e^{-2\pi i\xi\omega}V_gf(x-\xi, \omega+\frac{a\xi-\eta}{b}).
\end{align*}
In particular,
\begin{equation}
|Vg(T_\xi^AM_\eta^Af)(x,\omega)|=|V_gf(x-\xi,\omega+\frac{a\xi-\eta}{b})|.
\end{equation}
\end{proposition}
\begin{proof}
Consider
\begin{align*}
V_g(T_\xi^AM_\eta^Af)(x,\omega)&=\int_\R T_\xi^A M_\eta^Af(t)\overline{g(t-x)}e^{-2\pi i\omega t}dt\\
&=\int_\R e^{-\frac{2\pi ia}{b}\xi(t-\xi)}e^{\frac{\pi i}{b}(a\eta^2-2p\eta +2\eta(t-\xi))}f(t-\xi)\overline{g(t-x)}e^{-2\pi i\omega t}dt\\
&= e^{i\frac{2\pi a}{b}\xi^2}\rho_A(-\eta)e^{-\frac{2\pi i}{b}\eta \xi}\\
&\times \int_\R e^{-2\pi i\frac{a}{b}t\xi}e^{\frac{2\pi i}{b}\eta t}f(t-\xi)\overline{g(t-x)}e^{-2\pi it\omega}dt\\
&= e^{i\frac{2\pi a}{b}\xi^2}\rho_A(-\eta)e^{-\frac{2\pi i}{b}\eta \xi}\int_\R f(t)\overline{g(t+\xi-x)}e^{-2\pi i(t+\xi)(\omega + \frac{a\xi -\eta}{b})}dt\\
&=\rho_A(-\eta)e^{-2\pi i\xi \omega}\int_\R f(t)\overline{g(t+\xi-x)}e^{-2\pi it(\omega +\frac{a\xi-\eta}{b})}dt\\
&=\rho_A(-\eta)e^{-2\pi i\xi \omega}V_gf(x-\xi,\omega+\frac{a\xi-\eta}{b}),
\end{align*}
applying change of variables.
\end{proof}

\par It is well known that
\begin{align}\label{F.I.T.F.A.}
V_gf(x,\omega)=e^{-2\pi ix\omega}V_{\widehat{g}}\widehat{f}(\omega , -x),
\end{align}
which is popularly known as fundamental identity of time-frequency analysis. (See $(3.10)$ in \cite{grochenig_book}). We obtain an analogous result using the SAFT.

\begin{theorem}
We have
$$
V_{(\F_Ag)}\F_Af(x,\omega)= \eta_A(-x)\overline{\lambda}_A(dx-b\omega)e^{2\pi ip\frac{dx}{b}}e^{-2\pi i(x+p)\omega} V_gf(dx-b\omega,a\omega-cx).
$$
In particular,
\begin{equation}\label{eq:SAFT_STFT}
|V_{(\F_Ag)}\F_Af(x,\omega)| = |V_gf(dx-b\omega,a\omega-cx)|.
\end{equation}
\end{theorem}
\begin{proof}
First we observe that
\begin{align}\label{T_y STFT}
V_{(T_yg)}T_yf(x,\omega)=e^{-2\pi iy\omega}V_gf(x,\omega),\\
\label{D_y STFT} V_{(D_yg)}D_yf(x,\omega)=V_gf(x/y,y\omega).
\end{align}
Now consider
\begin{align*}
V_{(\F_Ag)}\F_Af(x,\omega)&=\int_\R \eta_A(t)\overline{\eta}_A(t-x)T_pD_b(C_\frac{a}{b}f)~\widehat{} \ (t)\overline{T_pD_b(C_\frac{a}{b}g)~\widehat{}~(t-x)}e^{-2\pi it\omega}dt\\
&=\overline{\eta}_A(-x)\int_\R T_pD_b(C_\frac{a}{b}f)~\widehat{} \ (t)\overline{T_pD_b(C_\frac{a}{b}g)~\widehat{}~(t-x)}e^{-2\pi it(\omega-\frac{dx}{b})}dt\\
&=\overline{\eta}_A(-x)V_{T_pD_b(C_\frac{a}{b}g)~\widehat{}~}T_pD_b(C_\frac{a}{b}f)~\widehat{}~(x,\omega-\frac{dx}{b})\\
&=\overline{\eta}_A(-x)e^{-2\pi ip(\omega-\frac{dx}{b})}V_{D_b(C_\frac{a}{b}g)~\widehat{}~}D_b(C_\frac{a}{b}f)~\widehat{}~(x,\omega-\frac{dx}{b})\\
&=\overline{\eta}_A(-x)e^{-2\pi ip(\omega -\frac{dx}{b})}V_{(C_\frac{a}{b}g)~\widehat{}~}(C_\frac{a}{b}f)~\widehat{}~(\frac{x}{b},b\omega -dx),\\
\end{align*}
using (\ref{T_y STFT}), (\ref{D_y STFT}). Then appealing to (\ref{F.I.T.F.A.}) we get
\begin{align*}
V_{(\F_Ag)}\F_Af(x,\omega)&=\overline{\eta}_A(-x)e^{-2\pi ip(\omega -\frac{dx}{b})}e^{-2\pi i\frac{x}{b}(b\omega -dx)}V_{C_\frac{a}{b}g}C_\frac{a}{b}f(dx-b\omega,\frac{x}{b})\\
&=\eta_A(-x)\overline{\lambda}_A(dx-b\omega)e^{2\pi ip\frac{dx}{b}}e^{-2\pi i(x+p)\omega}\\
&\times V_gf(dx-b\omega,\frac{x}{b}-\frac{a}{b}(dx-b\omega)),
\end{align*}
using Proposition \ref{Chirp_STFT}. But
$$
V_gf(dx-b\omega,-\frac{a}{b}(dx-b\omega))=V_gf(dx-b\omega,a\omega -cx),
$$
using $ad-bc=1$, from which the result follows.
\end{proof}

\begin{theorem}
Let $f\in \Msp_{v_\ell}^r$. Then $\F_Af\in \Msp_{w_\ell}^r$,
where $v_\ell=(1+x^2+\omega^2)^\frac{\ell}{2}$ and $w_\ell(x,\omega)=[1+(c^2+d^2)x^2+(a^2+b^2)\omega^2-
2(ac+bd)x\omega]^\frac{\ell}{2}.$
\end{theorem}
\begin{proof}
By (\ref{eq:SAFT_STFT}), we have $|V_{(\F_Ag)}\F_Af(x,\omega)|=|V_gf(dx-b\omega,a\omega-cx)|$, which implies that $|V_gf(u,v)|=|V_{(\F_Ag)}\F_Af(au+bv, cu+dv)|$.
Now
\begin{align*}
\|f\|_{\Msp_{v_\ell}^r}^r&=\int_\R\int_\R |V_gf(u,v)|^rv_\ell(u,v)^rdudv\\
&= \int_\R\int_\R |V_{(\F_Ag)}\F_Af(au+bv, cu+dv)|^rv_\ell(u,v)^rdu dv.
\end{align*}
Using again the transformation $x=au+bv,~\omega=cu+dv$, we get $\F_Af\in \Msp_{w_\ell}^r$.
\end{proof}

\textcolor{black}{\begin{corollary}
The modulation space $\Msp^r_{v_\ell}$ is invariant under the special affine Fourier transform $\F_A$.
\end{corollary}
\begin{proof}
It is enough to show that the weights $v_\ell(x,\omega)=(1+x^2+\omega^2)^{\ell/2}$ and $w_\ell(x,\omega)=(1+(c^2+d^2)x^2+(a^2+b^2)\omega^2-2(ac+bd)x\omega)^{\ell/2}$ are equivalent.
For
$$
M=
\begin{pmatrix}
1 & 0 & 0\\
0 & d & -b\\
0 & -c & a
\end{pmatrix}.
$$
one has $det(M)=det(A)=ad-bc=1$. Thus $M$ is invertible and   there exist $c_1, c_2>0$ such that
$$
c_1\Big \|\begin{pmatrix}
1\\
x\\
\omega
\end{pmatrix}\Big\|_2^2\leq \Big \|M\begin{pmatrix}
1\\
x\\
\omega
\end{pmatrix}\Big \|^2_2\leq c_2 \Big \|\begin{pmatrix}
1\\
x\\
\omega
\end{pmatrix}\Big \|^2_2.
$$
In other words,
$$
c_1(1+x^2+\omega^2)\leq (1+(a\omega-cx)^2+(dx-b\omega)^2) \leq c_2(1+x^2+\omega^2).
$$
The equivalence is finally established via the estimate
$$
c_1^{\ell/2}(1+x^2+\omega^2)^{\ell/2}\leq \big(1+(a\omega-cx)^2+(dx-b\omega)^2\big)^{\ell/2} \leq c_2^{\ell/2}(1+x^2+\omega^2)^{\ell/2}.
$$
\end{proof}}



\begin{theorem}
(i) The special affine Fourier transform $\F_A$ is a Gelfand triple automorphism on the Banach Gelfand triple $\big(S_0(\R),L^2(\R),S_0^\prime(\R) \big)$.\\
(ii) The pointwise multiplication operators defined using the auxiliary functions $\lambda_A,\rho_A,\eta_A$ are Gelfand triple automorphisms on $\big(S_0(\R),L^2(\R),S_0^\prime(\R) \big).$
\end{theorem}
\begin{proof}
Due to Lemma \ref{lemma:gelfand automorphism}, it is enough to prove that the operators are unitary on $L^2(\R)$ and their restrictions on $S_0(\R)$ define bounded and bijective mappings of $S_0(\R)$ onto itself. \\
(i) we know that $\F_A:L^2(\R)\to L^2(\R)$ is a unitary operator. Further, one can write $\F_A(f)(\omega)=C_{d/b}M_{q-dp/b}T_pD_b(C_\frac{a}{b}f)~\widehat{}~(\omega)$. Then using Remark \ref{rk:chirp invariance} it follows that $C_\frac{a}{b}$ is bounded and bijective on $S_0(\R)$. Since the Fourier transform is bounded and bijective on $S_0(\R), \ f\mapsto (C_\frac{a}{b}f)~\widehat{}$ is bounded and bijective on $S_0(\R)$. Further, we can easily show that
\begin{equation}\label{eq:M_y STFT}
V_{(M_yg)}M_yf(x, \omega)=e^{2\pi i xy}V_gf(x,\omega).
\end{equation}
Then using (\ref{T_y STFT}), (\ref{D_y STFT}), (\ref{eq:M_y STFT}) and proceeding as in Proposition \ref{Chirp_STFT} one can show that $T_y,M_y,D_y$ are bounded and bijective on $S_0(\R)$. Hence $\F_A:S_0(\R)\to S_0(\R)$ is bounded and invertible. Thus the result follows from Lemma \ref{lemma:gelfand automorphism}. \\
(ii) Since $|\lambda_A(\omega)|=|\rho_A(\omega)|=|\eta_A(\omega)|=1,$ the pointwise multiplication operators defined by $\lambda_A, \rho_A,\eta_A$ are unitary on $L^2(\R)$. Further,
\begin{align*}
&\lambda_A(x)f(x)=C_\frac{a}{b}f(x)\\
&\rho_A(x)f(x)=C_\frac{a}{b}M_\frac{p}{b}(x)\\
&\eta_A(x)f(x)=C_\frac{d}{b}M_\frac{\Omega}{b}f(x).
\end{align*}
This means that the multiplication operators defined by $\lambda_A,\rho_A,\eta_A$ are composition of $C_s$ and $M_y$. In part (i), we have already shown that $M_y$ is bounded and bijective on $S_0(\R)$. Further, by Remark \ref{rk:chirp invariance}, $C_s$ is bounded and bijective on $S_0(\R)$. Thus our assertion follows from Lemma \ref{lemma:gelfand automorphism}.
\end{proof}


\textcolor{black}{\begin{proposition}
The map $s\mapsto T_s^A$ is a strongly continuous and isometric projective representation on $S_0(\R)$.
\end{proposition}
\begin{proof}
It is easy to see that
$$
T_x^A\circ T_y^A=e^{-\frac{2\pi ia}{b}xy}T^A_{x+y}.
$$
Moreover, an easy computation shows that 
\begin{align}\label{eq:strong continuity}
V_gT_sf(x,\omega)&=e^{-2\pi i s\omega}V_gf(x-s,\omega)\\
V_gM_sf(x,\omega)&=V_gf(x,\omega -s)\label{eq:strong continuity2}.
\end{align} 
Thus $$\|T_sf\|_{S_0(\R)}=\|V_gT_sf\|_{L^1(\R^2)}=\|V_gf\|_{L^1(\R^2)}=\|f\|_{S_0(\R)},$$
using \eqref{eq:strong continuity}. Similarly, using \eqref{eq:strong continuity2} we can show that $\|M_s f\|_{S_0(\R)}=\|f\|_{S_0(\R)}.$ Consequently, $\|T_s^Af\|_{S_0(\R)}=\|T_sM_{-\frac{as}{b}}f\|_{S_0(\R)}=\|f\|_{S_0(\R)}$.
Now consider
\begin{align*}
\|T_s^Af-f\|_{S_0(\R)}&=\|V_gT_s^Af-V_gf\|_{L^1(\R^2)}\\
&=\int_\R \int_\R |V_gT_s^Af(x,\omega)-V_gf(x,\omega)|dx d\omega \\
&= \int_\R \int_\R |\int_\R e^{-\frac{2\pi ia}{b}s(t-s)}f(t-s)\overline{g(t-x)}e^{-2\pi i\omega t}dt-V_gf(x,\omega)|dx d\omega\\
&=\int_\R \int_\R |\int_\R e^{-\frac{2\pi ia}{b}st} f(t)\overline{g(t+s-x)}e^{-2\pi i\omega (t+s)}dt -V_gf(x,\omega)|dx d\omega\\
&=\int_\R \int_\R |e^{-2\pi is\omega}V_gf(x-s,\omega+\frac{as}{b})-V_gf(x,\omega)|dxd\omega\\
&\leq \int_\R \int_\R |e^{-2\pi is\omega}V_gf(x-s,\omega+\frac{as}{b})-e^{-2\pi is\omega}V_gf(x,\omega)|dxd\omega \\
&+\int_\R \int_\R |e^{-2\pi is\omega}-1|~ |V_gf(x,\omega)|dxd\omega\\
&=\int_\R \int_\R |T_{(s,-\frac{as}{b})}V_gf(x,\omega)-V_gf(x,\omega)|dxd\omega \\
&+\int_\R \int_\R |e^{-2\pi is\omega}-1|~ |V_gf(x,\omega)|dxd\omega,
\end{align*}
by applying change of variables, where $T_{(h,k)}f(x,y)=f(x-h,y-k)$. Finally, using $\|T_{(h,k)}f-f\|_{L^1(\R^2)}\to 0$ as $\|(h,k)\|\to 0$ and Lebesgue dominated convergence theorem we obtain $\|T_s^Af-f\|_{S_0(\R)} \to 0$ as $s\to 0$. In other words, the projective representation $s\mapsto T_s^A$ is strongly continuous.
\end{proof}}

\section{$A$-modulation spaces}

\par In this section, we wish to define the modulation spaces associated with the SAFT. We follow the classical definition of a modulation space as provided in \cite{FeichtingerLCA}.

\begin{definition}
Let $m(x,\omega)$ be a moderate weight \textcolor{black}{of polynomial growth} on $\R^2$ and $1\leq r,s<\infty$. Then, for $g\in S(\R)$, the modulation space associated with the SAFT, called $A$-modulation space, is defined as follows.
$$
\Msp_{A,m}^{r,s}=\{f\in S^\prime (\R):\|f\|_{\Msp_{A,m}^{r,s}}<\infty\},
$$
where $$\|f\|_{\Msp_{A,m}^{r,s}}=\left(\int_\R\big(\int_\R |f\star_AM_\omega^Ag(x)|^rm(x,\omega)^rdx\big)^\frac{s}{r}d\omega\right)^\frac{1}{s}.$$
\par If $r=s,$ then the modulation space $\Msp^{r,r}_{A,m}$ is denoted by $\Msp^r_{A,m}$. If $m(x,\omega)=1,$ then we write $\Msp^{r,s}_A$, $\Msp^r_A$ for $\Msp^{r,s}_{A,m}$, $\Msp^{r,r}_{A,m}$ and so on.
\end{definition}

\par We can rewrite $f\star_AM_\omega^Ag$ in terms of ordinary convolution using $C_\frac{a}{b}$:
\begin{proposition}
Let $f\in S^\prime(\R)$ and $g\in S(\R)$. Then
\begin{equation}\label{pro:convolution, A-convolution}
|f\star_AM_\omega^Ag(x)|=\frac{1}{\sqrt{|b|}}|C_\frac{a}{b}f\star M_{\frac{\omega}{b}}C_\frac{a}{b}g(x)|.
\end{equation}
\end{proposition}
\begin{proof}
Consider
\begin{align*}
(f\star_A M_\omega^Ag)(x)&=\int_\R f(y)T_y^AM_\omega^Ag(x)dy\\
&= \int_\R f(y)e^{-\frac{2\pi i a}{b}y(x-y)}M_\omega
^Ag(x-y)dy\\
&=\int_\R f(y) e^{-\frac{2\pi i a}{b}y(x-y)}e^{\frac{\pi i}{b}(a\omega^2-2p\omega+2\omega(x-y))}g(x-y)dy\\
&= \rho_A(-\omega)\int_\R f(y)e^{-\frac{2\pi i a}{b}y(x-y)}M_{\frac{\omega}{b}}g(x-y)dy\\
&=\rho_A(-\omega)\int_\R f(y)T_y^AM_{\frac{\omega}{b}}g(x)dy\\
&=\rho_A(-\omega)(f\star_AM_{\frac{\omega}{b}}g)(x).
\end{align*}
Thus, using (\ref{eq:A-convolution, convolution}) and the fact that $M_{\frac{\omega}{b}}C_\frac{a}{b}=C_\frac{a}{b}M_{\frac{\omega}{b}}$, we get
\begin{align*}
|(f\star_AM_\omega^Ag)(x)|&=|f\star_AM_{\frac{\omega}{b}}g(x)|\\
&=|C_\frac{a}{b}(f\star_AM_{\frac{\omega}{b}}g)(x)|\\
&= \frac{1}{\sqrt{|b|}}|C_\frac{a}{b}f\star M_{\frac{\omega}{b}}C_\frac{a}{b}g(x)|,
\end{align*}
proving our assertion.
\end{proof}


\par Now we give a relation between the new modulation space and the classical modulation space.

\begin{theorem}\label{A modulation}
Let $1\leq r,s<\infty$. Then $f\in \Msp_{A,m}^{r,s}$ if and only if $C_\frac{a}{b}f\in \Msp_{m_b}^{r,s},$ where $m_b(x,\omega)=m(x,b\omega).$
\end{theorem}
\begin{proof}
Consider
\begin{align*}
\|f\|_{\Msp_{A,m}^{r,s}}^s=&\int_\R \big(\int_\R |f\star_A M_\omega^Ag(x)|^rm(x,\omega)^rdx\big)^\frac{s}{r}d\omega\\
=& \frac{1}{|b|^{s/2}}\int_\R\big(\int_\R |C_\frac{a}{b}f\star M_{\frac{\omega}{b}}C_\frac{a}{b}g(x)|^rm(x,\omega)^rdx\big)^\frac{s}{r}d\omega,
\end{align*}
using (\ref{pro:convolution, A-convolution}). Then
$$\|f\|_{\Msp_{A,m}^{r,s}}^s=|b|^{1-s/2}\int_\R\big(\int_\R|C_\frac{a}{b}f\star M_\omega C_\frac{a}{b}g(x)|^rm(x,b\omega)^rdx\big)^\frac{s}{r}d\omega,$$
applying change of variables.
But
\begin{align*}
&c_1\big(\int_\R\big(\int_\R|C_\frac{a}{b}f\star M_\omega C_\frac{a}{b}g(x)|^rm(x,b\omega)^rdx\big)^\frac{s}{r}d\omega\big)^{1/s}\leq \|f\|_{\Msp_{m_b}^{r,s}}\\&\leq c_2\big(\int_\R\big(\int_\R|C_\frac{a}{b}f\star M_\omega C_\frac{a}{b}g(x)|^rm(x,b\omega)^rdx\big)^\frac{s}{r}d\omega\big)^{1/s},
\end{align*}
for some $c_1,c_2>0$, from which the result follows.
\end{proof}

\begin{corollary}
We have the following inclusion between $A$-modulation spaces.
$$
\Msp^{r_1,s_1}_A\subseteq \Msp^{r_2,s_2}_A, \ \ \text{for} \  r_1\leq r_2, \ s_1\leq s_2.
$$
In particular, $S_0(\R)\subseteq \Msp^{r,s}_A$ for $r,s\geq 1$.
\end{corollary}
\begin{proof}
Let $f\in \Msp^{r_1,s_1}_A$. Then $C_\frac{a}{b}f\in \Msp^{r_1,s_1}$. This implies that $C_\frac{a}{b}f\in \Msp^{r_2,s_2}$, from which it follows that $f\in \Msp^{r_2,s_2}_A.$
\end{proof}

\textcolor{black}{
\begin{corollary}
We have
 $$ \Msp^{r_1,s_1}_A \star_A  \Msp^{r_2,s_2}_A\subseteq \Msp^{r_0,s_0}_A,$$ where $\frac{1}{r_1}+\frac{1}{r_2}=1+\frac{1}{r_0}, \frac{1}{s_1}+\frac{1}{s_2}=\frac{1}{s_0}$.
 \end{corollary}
\begin{proof}
Let $f\in \Msp^{r_1,s_1}_A, g\in \Msp^{r_2,s_2}_A$. Then $C_\frac{a}{b}f\in \Msp^{r_1,s_1}$ and $C_\frac{a}{b} g\in \Msp^{r_2,s_2}$. Now using the fact that $\Msp^{r_1,s_1} \star  \Msp^{r_2,s_2}\subseteq \Msp^{r_0,s_0},$ we get $C_\frac{a}{b} f\star C_\frac{a}{b}g\in \Msp^{r_0,s_0}$. But $C_\frac{a}{b}f\star C_\frac{a}{b}g=\sqrt{|b|}C_\frac{a}{b}(f\star_A g)$. Hence $C_\frac{a}{b}(f\star_A g)\in \Msp^{r_0,s_0}$. In other words, $f\star_A g\in \Msp^{r_0,s_0}_A.$
\end{proof}}

\begin{corollary}
If $1\leq r,s <\infty, $ then $\Msp^{r,s}_A\subset S_0^\prime(\R)$.
\end{corollary}
\begin{proof}
Let $f\in \Msp^{r,s}_A$. Then $C_\frac{a}{b}f\in \Msp^{r,s}$. Now using the fact that $\Msp^{r,s}\subset S_0^\prime(\R)$, we get $C_\frac{a}{b}f\in S_0^\prime(\R)$. Thus it is enough to prove that $S_0^\prime(\R)$ is invariant under the operator $C_s$. Let $\Lambda \in S_0^\prime(\R)$. Then for $\phi \in S(\R)$, we get
$$
|C_s\Lambda(\phi)|=|\Lambda(C_s\phi)|\leq \|\Lambda\|\|C_s\phi\|_{S_0(\R)}\leq c^\prime\|\Lambda\|\|\phi\|_{S_0(\R)}, \ \ \text{for some $c^\prime>0$}.
$$
Hence the result follows.
\end{proof}

\begin{corollary}
The $A$-modulation space $\Msp^r_A$ coincides with the classical modulation space $\Msp^r$ in the sense of equivalent norm.
\end{corollary}
\begin{proof}
The proof follows from Proposition \ref{Chirp_STFT} and Theorem \ref{A modulation}.
\end{proof}

\begin{remark} \label{chirp inverse}
If $m(x,\omega)=1$, then the operator $C_\frac{a}{b}:\Msp^{r,s}_A\to \Msp^{r,s}$ is bounded and invertible. The inverse is given by $C_\frac{a}{b}^{-1}=C_{-\frac{a}{b}}.$ Moreover, using the fact that modulation is an isometry on $\Msp^{r,s},$ we get the invertiblity of $M_\frac{p}{b}C_\frac{a}{b}=\rho_A:\Msp^{r,s}_A\to \Msp^{r,s}$.
\end{remark}

As in the case of classical modulation spaces, we obtain the following
\begin{theorem}

\begin{itemize}
\item[(i)] 
 $\Msp^{r,s}_{A,m}$ is a Banach space for $1\leq r,s<\infty $.
\item[(ii)] The definition of modulation spaces is independent of
the choice of window $(0 \neq) g\in S(\R) $: 
different windows define the same space and equivalent norms.
\item[(iii)] $S(\R)$ is dense in $\Msp^{r,s}_A,$ for $1\leq r,s <\infty.$
\end{itemize}
\end{theorem}
We omit the proof as it follows from the corresponding properties of the classical modulation spaces.



\textcolor{black}{\begin{proposition}
Let $m$ be a $v$-moderate weight. Then the $A$-modulation space $\Msp^{r,s}_A$ is invariant under $A$-time-frequency shifts and
$$
\|M_\omega^A T_x^Af\|_{\Msp^{r,s}_{A,m}}\leq v(x,\omega)\|f\|_{\Msp^{r,s}_{A,m}}.
$$
\end{proposition}
\begin{proof}
Consider
\begin{align*}
M_\omega^A T_x^Af\star_AM_{\omega_1}^Ag(x_1)&=\int_\R T_y^AM_\omega^AT_x^Af(x_1)M_{\omega_1}g(y)dy\\
&=\int_\R \rho_A(-\omega)e^{-\frac{2\pi ia}{b}y(x_1-y)}e^{\frac{2\pi i\omega}{b}(x_1-y)}e^{-\frac{2\pi ia}{b}x(x_1-x-y)}\\
&\times f(x_1-x-y)\rho_A(-\omega_1)e^{\frac{2\pi i\omega_1}{b}y}g(y)dy\\
&=\rho_A(-\omega)\rho_A(-\omega_1)e^{\frac{2\pi i}{b}(\omega x_1-axx_1+ax^2)}\\
&\times \int_\R e^{-\frac{2\pi ia}{b}y(x_1-x-y)}f(x_1-x-y)M_\frac{\omega_1-\omega}{b}g(y)dy\\
&=\rho_A(-\omega)\rho_A(-\omega_1)\overline{\rho_A(\omega-\omega_1)}e^{\frac{2\pi i}{b}(\omega x_1-axx_1+ax^2)}\\
&\times \int_\R T_y^Af(x_1-x)\rho_A(\omega-\omega_1)M_\frac{\omega_1-\omega}{b}g(y)dy\\
&=e^{\frac{2\pi i}{b}(\omega x_1-axx_1+ax^2+a\omega \omega_1-2p\omega)}\int_\R T_y^Af(x_1-x)M_{\omega_1-\omega}^Ag(y)dy\\
&=e^{\frac{2\pi i}{b}(\omega x_1-axx_1+ax^2+a\omega \omega_1-2p\omega)}
 f\star_AM_{\omega_1-\omega}^Ag(x_1-x).
\end{align*}
Thus
$$
|M_\omega^AT_x^Af\star_AM_{\omega_1}^Ag(x_1)|=|f\star_AM_{\omega_1-\omega}^Ag(x_1-x)|.
$$
Then
\begin{align*}
\|M_\omega^AT_x^Af\|_{\Msp^{r,s}_{A,m}}&=\Big(\int_\R \Big(\int_\R |M_\omega^AT_x^Af\star_AM_{\omega_1}^Ag(x_1)|^rm(x_1,\omega_1)^rdx_1\Big)^{s/r}d\omega_1 \Big)^{1/s}\\
&=\Big( \int_\R \Big( |f\star_AM_{\omega_1-\omega}^Ag(x_1-x)|^rm(x_1,\omega_1)^r \Big)^{s/r} \Big)^{1/s}\\
&=\Big( \int_\R \Big( |f\star_AM_{\omega_1}^Ag(x_1)|^rm(x_1+x,\omega_1+\omega)^r \Big)^{s/r} \Big)^{1/s}\\
&\leq v(x,\omega)\Big( \int_\R \Big( |f\star_AM_{\omega_1}^Ag(x_1)|^rm(x_1,\omega_1)^r \Big)^{s/r} \Big)^{1/s}\\
&= v(x,\omega)\|f\|_{\Msp^{r,s}_{A,m}},
\end{align*}
by applying a change of variables.
\end{proof}}

\textcolor{black}{\begin{corollary}
The $A$-modulation space $\Msp^{r,s}_{A,v}$ is invariant under classical time-frequency shifts and
$$
\|M_\omega T_xf\|_{\Msp^{r,s}_{A,v}}\leq v(x, ax+b\omega )\|f\|_{\Msp^{r,s}_{A,v}}.
$$
\end{corollary}
\begin{proof}
Consider
$$
\|M_\omega T_xf\|_{\Msp^{r,s}_{A,v}}=\|M^A_{ax+b\omega }T_x^Af\|_{\Msp^{r,s}_{A,v}}\leq v(x, ax+b\omega )\|f\|_{\Msp^{r,s}_{A,v}},
$$
using $M_\omega T_x=e^{-\frac{2\pi ia}{b}x^2}\overline{\rho_A(-ax-b\omega)}M^A_{ax+b\omega}T_x^A$.
\end{proof}}

\textcolor{black}{\begin{theorem}
Let $\B$ be a Banach space  continuously embedded into the space of tempered distributions with the following properties:
\begin{itemize}
\item[(i)] $\B$ is invariant under $A$-time-frequency shifts and $\|M_\omega^AT_x^Af\|_\B\leq v(x,\omega)\|f\|_\B,$ for all $f\in \B.$
\item[(ii)] $\Msp_{A,v}^1\bigcap \B\neq \{0\}$.
\end{itemize}
Then $\Msp_{A,v}^1$ is continuously embedded in $\B$.
\end{theorem}
\begin{proof}
Choose a non-zero $g\in \Msp_{A,v}^1\bigcap \B$. Then $C_\frac{a}{b}g\in \Msp^1_{v_b}$, where $v_b(x,\omega)=v(x, b\omega)$. Let $f\in \Msp_{A,v}^1$. Then $C_\frac{a}{b}f\in \Msp^1_{v_b}$. Now using Theorem 12.1.8 in \cite{grochenig_book}, we can express $C_\frac{a}{b}f$ as a non-uniform Gabor expansion of time-frequency shifts of $C_\frac{a}{b}g$. In other words,
\begin{equation}\label{eq:minimality1}
C_\frac{a}{b}f=\sum_{n=1}^\infty c_nT_{x_n}M_{\frac{\omega_n}{b}}C_\frac{a}{b}g,
\end{equation}
with $\|C_\frac{a}{b}f\|_{\Msp^1_{v_b}}=\displaystyle \inf\sum_{n=1}^\infty |c_n| v_b(x_n,\frac{\omega_n}{b})=\inf\sum_{n=1}^\infty |c_n| v(x_n,\omega_n)$, where the infimum is taken over all such representations of $C_\frac{a}{b}f$. Using \eqref{eq:minimality1} $f$ can be expressed as
$$
f=\sum_{n=1}^\infty c_n C_\frac{a}{b}^{-1}T_{x_n}M_{\frac{\omega_n}{b}}C_\frac{a}{b}g.
$$
Consider
\begin{align*}
C_\frac{a}{b}^{-1}T_{x_n}M_{\frac{\omega_n}{b}}C_\frac{a}{b}g(t)&=e^{-\frac{\pi ia}{b}t^2}e^{2\pi i\frac{\omega_n}{b}(t-x_n)}e^{\frac{\pi ia}{b}(t-x_n)^2}g(t-x_n)\\
&=e^{-\frac{\pi ia}{b}x_n^2}e^{2\pi i\frac{\omega_n}{b}(t-x_n)}e^{-\frac{2\pi ia}{b}x_n(t-x_n)}g(t-x_n)\\
&=e^{-\frac{\pi ia}{b}x_n^2}e^{-\frac{2\pi i}{b}x_n\omega_n}\overline{\rho_A(-\omega_n)}\rho_A(-\omega_n)M_\frac{\omega_n}{b}T_{x_n}^Ag(t)\\
&=\overline{\lambda(x_n+\omega_n)}e^{\frac{2\pi ia}{b}p\omega_n}M_{\omega_n}^AT_{x_n}^Ag(t).
\end{align*}
Thus
$$
f=\displaystyle\sum_{n=1}^\infty c_n \overline{\lambda(x_n+\omega_n)}e^{\frac{2\pi ia}{b}p\omega_n}M_{\omega_n}^AT_{x_n}^Ag(t).
$$
Then
\begin{equation}\label{eq:minimality3}
\|f\|_\B \leq \displaystyle \sum_{n=1}^\infty |c_n|\|M_{\omega_n}^AT_{x_n}^Ag\|_\B \leq \sum_{n=1}^\infty |c_n|v(x_n,\omega_n)\|g\|_\B < \infty.
\end{equation}
This shows that $\Msp^1_{A,v}\subseteq \B$. Taking infimum over all representations of $C_\frac{a}{b}f$, \eqref{eq:minimality3} turns out to be
$$
\|f\|_\B\leq \|C_\frac{a}{b}f\|_{\Msp^1_{v_b}}\|g\|_\B\leq C \|f\|_{\Msp^1_{A,v}}, \ \ \text{ for some $C>0$},
$$
using Theorem \ref{A modulation}. Hence the inclusion is continuous.
\end{proof}}

\section{Multipliers and Littlewood-Paley theorem associated with SAFT}

First we prove an analogue of Theorem 10 in \cite{FeiAcha06}.

\begin{theorem}
Given a bounded linear operator
 $T:\Msp^{r_1,s_1}_A\to \Msp^{r_2,s_2}_A$, one has 
 the following.
\begin{itemize}
\item[(i)] If $TT_x^A=T_x^AT$ for all $x\in \R$, then there exists a unique $u\in S^\prime_0(\R)$ such that $Tf=u\star_Af$, for all $f\in S_0(\R)$.

\item[(ii)] If $TT_x=T_xT$ for all $x\in \R$, then there exists a unique $u\in S^\prime_0(\R)$ such that
$Tf=u\star f$, for all $f\in S_0(\R).$
\end{itemize}
\end{theorem}
\begin{proof}
(i) Define $\tilde{T}:\Msp^{r_1,s_1}\to \Msp^{r_2,s_2}$ by $\tilde{T}=C_\frac{a}{b}TC_\frac{a}{b}^{-1}.$ Then $\tilde{T}$ is bounded and linear. Further, using $TT_x^Af=T_x^ATf$ and (\ref{eq:T_xT_x^A}), we get $$Te^{\frac{\pi ia}{b}x^2}C^{-1}_{\frac{a}{b}}T_xC_{\frac{a}{b}}f=e^{\frac{\pi ia}{b} x^2}C^{-1}_{\frac{a}{b}}T_xC_{\frac{a}{b}}Tf.$$ This is equivalent to
$$
C_{\frac{a}{b}}TC^{-1}_{\frac{a}{b}}T_xC_{\frac{a}{b}}f=T_xC_{\frac{a}{b}}TC_{\frac{a}{b}}^{-1}C_{\frac{a}{b}}f.
$$
In other words, $\tilde{T}T_xC_\frac{a}{b}f=T_x\tilde{T}C_\frac{a}{b}f$, for $f\in \Msp^{r_1,s_1}_A$, from which it follows that $\tilde{T}$ commutes with classical translations.
Then by Theorem 10 in \cite{FeiAcha06}, there exists a unique $u\in S_0^\prime(\R) $ such that
$
\tilde{T}f=u\star f$, for all $f\in S_0(\R)$. This means that
$$
C_\frac{a}{b}TC_\frac{a}{b}^{-1}f=C_\frac{a}{b}C_\frac{a}{b}^{-1}u\star C_\frac{a}{b}C_\frac{a}{b}^{-1}f=\sqrt{|b|}C_\frac{a}{b}(C_\frac{a}{b}^{-1}u\star_A C_\frac{a}{b}^{-1}f),
$$
using Proposition \ref{Pro:tempered chirp} and Remark \ref{rk:S_0 chirp}. Define $u_1=\sqrt{|b|}C_\frac{a}{b}^{-1}u$, which belongs to  $S_0^\prime(\R).$ Thus
$$
C_\frac{a}{b}TC_\frac{a}{b}^{-1}f=C_\frac{a}{b}(u_1\star_AC_\frac{a}{b}^{-1}f),
$$
from which our assertion follows.\\
(ii) We omit the proof as it is using the same ideas as in the proof of Theorem 10 in \cite{FeiAcha06},  with $\Msp^{r,s}_A$ in place of $\Msp^{r,s}$.
\end{proof}


\par Recall that $m\in L^\infty(\R)$ is called a Fourier multiplier for $L^r(\R),~1\leq r<\infty$, if the operator defined by $(T_mf)~\widehat{}~(\xi)=m(\xi)\widehat{f}(\xi)$ on $L^2(\R)\cap L^r(\R)$, extends to a bounded linear operator on $L^r(\R)$.

\begin{definition}
Let $m\in L^\infty (\R ).$ Then $m$ is called a SAFT multiplier for $L^r(\R), 1\leq r<\infty$ if the operator $T_{A,m}$, defined by
$$
\F_A(T_{A,m}f)(\omega)=m(\omega)\F_A(f)(\omega), \ \ f\in L^2(\R)\cap L^r(\R).
$$
extends to a bounded linear operator on $L^r(\R)$. \end{definition}

\begin{theorem}(H\"{o}rmander)
Let $m\in C^1(\R \setminus \{0\})$ satisfy $|m^\prime (x)|\leq C|x|^{-1}$ for some $C>0$. Then for any $1<r<\infty$, $m$ is a SAFT multiplier for $L^r(\R)$.
\end{theorem}

\begin{proof}
We first observe that
$$\F_A(T_{A,m}f)(\omega)=m(\omega)\F_A(f)(\omega)$$ is equivalent to $$\frac{\eta_A(\omega)}{\sqrt{|b|}} (\rho_AT_{A,m}f)~\widehat \ (\omega/b)=\frac{\eta_A(\omega)}{\sqrt{|b|}}m(\omega)(\rho_Af)~\widehat \ (\omega/b).$$
This in turn implies that $$(\rho_AT_{A,m}\bar \rho_Af)~\widehat ~ (\omega)=m(b\omega)\widehat{f}(\omega).$$
Further, let $m_1(x)=m(bx)$. Then $m_1^\prime(x)=bm^\prime(bx)$. Hence $$|m_1^\prime (x)|\leq |b|C|bx|^{-1}=C|x|^{-1}.$$ Now by applying classical H\"{o}rmander theorem to $m_1$, it follows that $m_1$ is a Fourier multiplier for $L^r(\R)$. Thus we can find $C^\prime >0$ such that
$$
\|\rho_AT_{A,m}\bar \rho_Af\|_r\leq C^\prime \|f\|_r.
$$
In other words,
$$
\|T_{A,m}\bar \rho_Af\|_r\leq C^\prime \|\bar \rho_Af\|_r,
$$
which leads to $\|T_{A,m}f\|_r\leq C^\prime \|f\|_r$, proving our assertion.
\end{proof}

\begin{theorem}(Littlewood-Paley)
For any  $1 < r < \infty$   there exist $m_r,\ M_r>0$ such that
$$
m_r\|f\|_r\leq \|(\sum_{j\in \Z}|S_jf|^2)^{1/2}\|_r\leq M_r\|f\|_r,
\quad f \in L^r(\R),
$$
with 
$\F_A(S_jf)(\omega)=\bigchi_{\Delta_j}(\omega)\F_A(f)(\omega)$, for
$\Delta_j=[-2^{j+1},-2^j]\cup [2^j,2^{j+1}].$
\end{theorem}
\begin{proof}
First we observe that
\[\frac{1}{b}\Delta_j=
\begin{cases}
\big(-\frac{2^{j+1}}{b},-\frac{2^j}{b}\big]\bigcup \big[\frac{2^j}{b},\frac{2^{j+1}}{b}\big), &  \ b>0\\
\big[-\frac{2^j}{b},-\frac{2^{j+1}}{b}\big)\bigcup \big(\frac{2^{j+1}}{b},\frac{2^j}{b}\big], & \ b<0
\end{cases}.
\]
Further we can see that
$$ \F_A(S_jf)(\omega)=\bigchi_{\Delta_j}(\omega)\F_A(f)(\omega),$$ 
which is equivalent to
 $$\frac{\eta_A(\omega)}{\sqrt{|b|}}(\rho_AS_jf)~\widehat ~ (\omega/b)=\bigchi_{\Delta_j}(b\omega)\frac{\eta_A(\omega)}{\sqrt{|b|}}(\rho_Af)~\widehat ~ (\omega/b).$$ In other words,
  $$(\rho_AS_jf)~\widehat{}~(\omega)=\bigchi_{\Delta_j}(b\omega)(\rho_Af)~\widehat{}~(\omega),$$ which leads to
 $$(\rho_AS_j\bar \rho_Af)~\widehat \ (\omega)=\bigchi_{\frac{1}{b}\Delta_j}(\omega)\widehat{f}(\omega).$$

Let $\tilde{S}_j=\rho_AS_j\bar \rho_A$. Then $(\tilde{S}_jf)~\widehat ~ (\omega)=\bigchi_{\frac{1}{b}\Delta_j}(\omega)\widehat{f}(\omega).$ Now by applying classical Littlewood-Paley theorem to the sequence of intervals $\{\frac{1}{b}\Delta_j:j\in \Z\}$ we can find $m_r,M_r>0$ such that\\
$$
m_r\|f\|_r\leq \|\big(\sum_{j \in \Z}|\tilde{S}_jf|^2\big)^{1/2}\|_r\leq M_r\|f\|_r.$$ In other words,
$$m_r\|\bar{\rho}_Af\|_r\leq \|\big(\sum_{j\in \Z}|\rho_AS_j\bar{\rho}_Af|^2\big)^{1/2}\|_r\leq M_r \|\bar \rho_Af\|_r,
$$ which in turn implies that
$$m_r\|f\|_r\leq \|\big(\sum_{j\in \Z}|S_jf|^2\big)^{1/2}\|_r\leq M_r\|f\|_r.$$
\end{proof}


\textbf{Data availability statement}\\ 
\ \ The manuscript has no associated data.

\bibliography{references}
\bibliographystyle{amsplain}

\end{document}